\documentclass[11pt]{article}
\usepackage[margin=1in]{geometry}
\usepackage{amsmath,amssymb,amsthm,booktabs}
\usepackage[expansion=false]{microtype}
\usepackage{xurl}
\usepackage[colorlinks=true,linkcolor=blue,citecolor=blue,urlcolor=blue]{hyperref}
\newtheorem{theorem}{Theorem}
\newtheorem{lemma}[theorem]{Lemma}
\newtheorem{proposition}[theorem]{Proposition}
\theoremstyle{remark}\newtheorem{remark}[theorem]{Remark}
\newcommand{\Rea}{\operatorname{Re}}
\newcommand{\Ima}{\operatorname{Im}}
\newcommand{\e}{\varepsilon}
\newcommand{\Log}{\operatorname{Log}}
\title{A conditional bound $\Lambda\le0.13963$\\
for the de Bruijn--Newman constant}
\author{Cetin Hakimoglu-Brown}
\date{Working draft: 7 October 2026}
\begin{document}
\maketitle
\begin{abstract}
Assuming that $\zeta(s)$ has no nontrivial zero with $\Rea s>7/8$,
as asserted in a recent unrefereed preprint of OpenAI~\cite{O},
we obtain $\Lambda\le13963/100000=0.13963$ for the de Bruijn--Newman
constant. A published Polymath boundary computation supplies a
finite-region strip of half-width $1/5$. We derive an approximation
on $3\le\Ima z\le6$ from its general Riemann--Siegel coefficient
bounds. At height $61/20$, a fixed cutoff permits two differentiations,
and comparison with $\zeta$ bounds both the zero-density kernel and
its vertical derivative. A new conjugate-pair inequality and a
twenty-four-step strip comparison give the stated bound.
The finite scalar and polynomial certificates are included in the paper;
earlier comparisons giving $0.17973$, $0.17172$, and $0.16408$
are retained to show the source of the gain.
The argument uses neither Gordon's analytic inputs nor Gomila's new
computations. The seven-eighths hypothesis remains an assumption,
and the published numerical boundary input is retained.
\end{abstract}

\section{Scope and statement}
Use the normalization of Polymath~\cite{P}:
\[
 H_t(z)=\int_0^\infty e^{tu^2}\Phi(u)\cos(zu)\,du,
 \qquad H_0(z)=\tfrac18\xi(\tfrac12+iz/2).
\]
Write $S(x,p,t)=-\Ima(H_t'/H_t)(x+ip)$ and
$M(t)=\sup\{(\Ima z)^2:H_t(z)=0\}$. Set
\[
 X=60000083951.5,\qquad b=\frac1{25}.
\]
Every decimal used in an inequality denotes an exact rational number.

\begin{theorem}\label{thm:main}
Assume every nontrivial zero of $\zeta$ has real part at most $7/8$.
Then
\[
 \Lambda\le\frac{13963}{100000}=0.13963.
\]
\end{theorem}

The hypothesis is the zeta case of the main theorem of~\cite{O};
Section~\ref{sec:hypothesis} records the source, its status, and the
single place where it enters the argument.

Platt and Trudgian proved the unconditional bound $\Lambda\le0.2$
in \cite[Corollary 2]{PT}, improving Polymath's $0.22$.
Theorem~\ref{thm:main} improves $0.2$ by $0.06037$ under the
additional half-plane assumption. Earlier versions of the comparison
gave $0.1938798366\ldots$ at height $5$ and
$0.1797275414\ldots$ at height $4$. Retaining the time dependence
of the heat weights, using the Euler product for the nearby zeta
function, and subdividing time gave $0.16408$. The further gain here
comes from a kernel comparison that also uses the vertical derivative
of $S$. It requires no extension of the approximation band.

Gomila~\cite{Go} claims an unconditional bound $\Lambda\le0.1787854$
in an unrefereed preprint. Its new interval
computations are not inputs to this paper. Our smaller numerical
constant still requires the seven-eighths hypothesis.

The source inputs are the published zero dynamics, the finite RH
verification used by Polymath, its boundary computation in Section 8.4,
and its generic-real-part approximation formulas. The proof below still
uses published numerical work; it is not a computation-free argument.
The seven-eighths assertion of~\cite{O} remains an assumption and its
source proof is not assessed here. This working note records a new derivation for
mathematical review, rather than a claim of independent refereeing.

The basic inequality $K_h\ge(h/p)K_p$ used below is also in Gordon's
Section 5; the derivative extension in Section~\ref{sec:kernel} is
proved directly here.
The strategy of controlling a logarithmic derivative is shared with his
work, and our treatment of the Gaussian remainder was informed by his
Appendix A. No theorem or numerical certificate from that appendix is
assumed here: the approximation estimates needed on the band $3\le y\le6$ are rederived with
coarser constants from the published coefficient bounds.

\subsection{The seven-eighths hypothesis: source and use}
\label{sec:hypothesis}
\emph{Source.} On 6 October 2026 OpenAI released a collection of
manuscripts produced by an internal model, in the public repository
\texttt{openai/math}. One of them, dated 30 September 2026 and titled
\emph{The Quasi-Riemann Hypothesis: A Zero-Free Half-Plane
$\Rea(s)>7/8$}~\cite{O}, states as its main theorem that $\zeta(s)$,
every Dirichlet $L$-function, and every finite-order Hecke
$L$-function over $\mathbb Q(\sqrt{-3})$ have no zero with
$\Rea s>7/8$, apart from the pole at $s=1$. The repository also
describes a Lean formalization of the zeta and Dirichlet statements.

\emph{Status.} The manuscript is unrefereed and, at the time of
writing, has not been confirmed by independent mathematicians. Neither
its proof nor its formalization has been examined for the present
paper. Theorem~\ref{thm:main} is therefore a conditional statement,
and it is phrased with the hypothesis written out so that it does not
depend on the fate of~\cite{O}.

\emph{What is used.} Only the zeta case is used:
\begin{equation}\label{eq:H78}
 \zeta(\rho)=0,\ 0<\Rea\rho<1\quad\Longrightarrow\quad\Rea\rho\le\tfrac78 .
\end{equation}
The statements about Dirichlet and Hecke $L$-functions and about
Landau--Siegel zeros play no role.

\emph{How it is used.} The functional equation $\xi(s)=\xi(1-s)$
turns \eqref{eq:H78} into $\tfrac18\le\Rea\rho\le\tfrac78$. Since
$H_0(z)=\tfrac18\xi(\tfrac12+iz/2)$, a zero $z=x+iy$ of $H_0$
corresponds to $\rho=\tfrac12-\tfrac y2+\tfrac{ix}2$, so
$|y|=|2\Rea\rho-1|\le\tfrac34$. Hence
\begin{equation}\label{eq:M0}
 M(0)\le\Bigl(\frac34\Bigr)^2=\frac9{16}.
\end{equation}
This is the only point at which the hypothesis enters. It fixes the
initial value $q_0=9/16$ of the strip comparison in
Theorem~\ref{thm:derivative-step}; the zero dynamics, the boundary
input, the approximation on $3\le y\le6$, the kernel inequality, and
all certificates are independent of it. By de Bruijn's strip theorem,
\eqref{eq:M0} alone gives $\Lambda\le9/32=0.28125$, which is weaker
than the unconditional $0.2$; the improvement to $0.13963$ comes from
combining \eqref{eq:M0} with the comparison developed below.

\emph{Other half-planes.} If the nontrivial zeros of $\zeta$ satisfy
$\Rea\rho\le\sigma$ for some $\tfrac12\le\sigma<1$, the same
computation gives $M(0)\le(2\sigma-1)^2$, and the comparison theorems
apply with $q_0=(2\sigma-1)^2$. The tables in this paper are certified
only for $\sigma=7/8$. In particular, the weaker half-plane
$\Rea s>11/12$ treated in a companion manuscript of the same release
would give $q_0=25/36$, for which no table is certified here.

\section{A published replacement for finite-region reality}
Polymath's Section 8.4, claim (a), establishes
\begin{equation}\label{eq:wall}
 H_t(x+iy)\ne0\quad
 (X\le x\le X+1,\quad 1/5\le y\le1,\quad0\le t\le1/5).
\end{equation}
Its finite RH input covers $0\le x\le X+1$ at $t=0$.
Both facts are stated and used in Sections 8.1--8.4 of~\cite{P}.
Specifically, $X=X_0-1/2$ with $X_0=6\cdot10^{10}+83952$,
and Platt's verified zeta height $3.06\cdot10^{10}$ exceeds
$(X+1)/2$. We use the published boundary conclusion; its numerical
certificate has not been replayed here.\footnote{In \cite[equation
(91)]{P}, the time-interpolation term should be $D_t|t'-t|$, by the
triangle inequality, rather than $D_t|t'-t|/(2n)$. The time-step update
in the code at \cite{PC} uses $(\min|f_t|-1/2)/D_t$ and has no mesh
divisor. This inspection concerns that formula only and is not a
verification of the full computation.}

\begin{lemma}\label{lem:head}
For $0\le t\le1/5$, a zero of $H_t$ with $|\Rea z|\le X$ satisfies
$|\Ima z|<1/5$.
\end{lemma}
\begin{proof}
Initially all zeros in this finite region are real. The classical strip
bound gives $|\Ima z|\le1$ throughout the time interval. If the
conclusion fails, consider the closed set of zero-time pairs with
$0\le x\le X$, $1/5\le y\le1$, and $0\le t\le1/5$.
It is compact and disjoint from $t=0$, so its least time $t_1>0$
is attained. No zero can enter across $x=X$ at height at least $1/5$,
by \eqref{eq:wall}; the same holds at $x=-X$ by symmetry.
Continuity of zeros therefore forces a first contact at height
$h=1/5$. We do not need continuity of the unrestricted maximum over
a window whose vertical sides other zeros can cross.
At a first contact in $0\le x\le X$, the right
boundary is excluded by \eqref{eq:wall}. The imaginary axis is zero-free
because $H_t(iy)$ is an integral of a positive integrand.
A multiple contact is excluded by backward Hermite splitting as in
\cite[Proposition 3.1(ii)]{P}: a zero of multiplicity $m\ge2$ at
$z_1=x_1+ih$ has, at time $t_1-v$, a neighbouring zero
$z_1+i\sqrt{2v}\lambda_m+O(v)$, where $\lambda_m>0$ is the largest
root of the probabilists' Hermite polynomial. Its imaginary part is
larger than $h$ for small $v>0$, contradicting first contact. Because
the contact is away from the vertical sides, this earlier zero remains
inside the finite region.

For a simple contact zero $x+ih$, its own conjugate contributes $-1/h$
to its vertical velocity. Every other conjugate pair contributes $-2K_h$,
where
\[
 K_h(d,v)=\frac{h-v}{d^2+(h-v)^2}
          +\frac{h+v}{d^2+(h+v)^2}
 =\frac{2h(d^2+h^2-v^2)}
        {[d^2+(h-v)^2][d^2+(h+v)^2]}.
\]
Pairs with $v\le h$ have nonnegative contribution to $K_h$.
A pair with $v>h$ lies outside $[-X-1,X+1]$ by first contact and
\eqref{eq:wall}, so $|d|\ge1$ and $v\le1$ imply $K_h>0$.
Real zeros carry half-pair weight. Thus $h'\le-1/h<0$, contradicting
first contact from below. Compactness of the finite rectangle and
continuity of zeros justify attainment of a first contact.
\end{proof}

This is weaker than Gordon's finite-region reality statement, but it
suffices if the global comparison stops at squared height $b=1/25$.

\section{A coarse approximation on a band above the zeros}
For $z=x+iy$ define
\begin{gather*}
 w=(1-iz)/2,\quad v=1-w,\quad \mathcal Q=x/(4\pi),\quad L=\log\mathcal Q,\\
 M_0(s)=\frac{s(s-1)}{16}\pi^{-s/2}\sqrt{2\pi}
  \exp\bigl((s/2-1/2)\Log(s/2)-s/2\bigr),\\
 \alpha(s)=\frac1{2s}+\frac1{s-1}+\frac12\Log\frac{s}{2\pi},
 \qquad M_t(s)=M_0(s)e^{t\alpha(s)^2/4},\\
 B_t(z)=M_t(w),\quad s_B=w+\tfrac t2\alpha(w),\quad
 s_A=v+\tfrac t2\alpha(v),\\
 \gamma=M_t(v)/M_t(w),\qquad
 N=\left\lfloor\sqrt{\mathcal Q+t/16}\right\rfloor,\\
 P_N(s,t)=\sum_{n=1}^N e^{t\log^2n/4}n^{-s}.
\end{gather*}
These normalizers are holomorphic and nonzero in the domains used below.

\begin{lemma}\label{lem:approx}
If $x\ge6\cdot10^{10}$, $0\le t\le1/5$, and $3\le y\le6$, then
\begin{equation}\label{eq:approx}
 \left|\frac{H_t(z)}{B_t(z)}-
       P_N(s_B,t)-\gamma P_N(s_A,t)\right|<10^{-6}.
\end{equation}
On the subband $4\le y\le6$, one also has
\begin{equation}\label{eq:disk}
 \left|\frac{H_t(z)}{B_t(z)}-1\right|<\frac7{20}.
\end{equation}
\end{lemma}
\begin{proof}
We use the exact heat-convolution decomposition in
\cite[equations (39)--(41)]{P} and its Proposition 6.2 for the
Riemann--Siegel coefficients at arbitrary real part. For the summands
we recover the bound of its Proposition 6.1 using a vertical Stirling
estimate. We do not apply its Theorem 1.3 or Proposition 6.3 at these
larger heights.

\emph{Source formulas and their domains.}
Here we restate the inputs with their original equation numbers.
For $g^*(s)=\overline{g(\overline s)}$, equation (39) of~\cite{P} is
\begin{equation}\label{eq:decomposition}
 H_t(z)=\sum_{n\le N}r_{t,n}(v)+\sum_{n\le N}r_{t,n}^*(w)
             +R_{t,N}(v)+R_{t,N}^*(w).
\end{equation}
It holds for $t>0$, $\Rea z\ne0$, and every integer $N\ge0$.
Writing
\[
 \mathcal M(s)=\frac{s(s-1)}{16}\pi^{-s/2}\Gamma(s/2),\qquad
 I_N(s)=\int_{N\swarrow N+1}
       \frac{u^{-s}e^{i\pi u^2}}{e^{\pi iu}-e^{-\pi iu}}\,du,
\]
equations (37)--(38) give $r_{0,n}(s)=\mathcal M(s)n^{-s}$ and
$R_{0,N}(s)=\mathcal M(s)I_N(s)$. The contour is directed southwest
across $(N,N+1)$. Equations (40)--(41) say that, for either of these
functions $g_0$ and its Gaussian evolution $g_t$,
\begin{equation}\label{eq:shifted-heat}
 g_t(s)=e^{-t\eta^2/4}\int_{\mathbb R}e^{-\sqrt t V\eta}
  g_0(s+\sqrt t V+t\eta/2)\frac{e^{-V^2}}{\sqrt\pi}\,dV,
\end{equation}
provided $\Ima s$ and $\Ima(s+t\eta/2)$ have the same nonzero sign.

For real $u$, $T'>0$, $a=\sqrt{T'/(2\pi)}$, $N=\lfloor a\rfloor$,
and $\rho=1-2(a-N)$, Proposition 6.2 gives, for every integer $K\ge1$,
\begin{gather}
 I_N(u+iT')=\omega_N a^{-u}
     \left(\sum_{k=0}^K\frac{C_k(\rho,u)}{a^k}
                           +RS_K(u+iT')\right),\label{eq:RS-source}\\
 \omega_N=(-1)^{N-1}
   \exp\left[-i\left(\frac{T'}2\log\frac{T'}{2\pi}
                         -\frac{T'}2-\frac\pi8\right)\right].\notag
\end{gather}
Its equations (53) and (58) give
\[
 C_0(\rho,u)=C_0(\rho)=
 \frac{e^{\pi i(\rho^2/2+3/8)}-i\sqrt2\cos(\pi\rho/2)}
      {2\cos(\pi\rho)},\qquad |C_0(\rho)|\le\tfrac12,
\]
with removable values at $\rho=\pm1/2$. Its equations (54)--(57) are
\begin{align}
 |C_k(\rho,u)|&\le\frac{\sqrt2}{2\pi}
            \frac{9^u\Gamma(k/2)}{2^k} &&(u>0,\ k\ge1),\label{eq:C-positive}\\
 |C_k(\rho,u)|&\le P_0\,2^{-u}\Gamma(k/2)\kappa_0^k
                  &&(u\le0,\ k\ge1),\label{eq:C-negative}\\
 |RS_K(u+iT')|&\le\tfrac17\,2^{3u/2}\Gamma((K+1)/2)(1.1/a)^{K+1}
                  &&(u\ge0),\label{eq:RS-positive}\\
 |RS_K(u+iT')|&\le\tfrac12(9/10)^{\lceil-u\rceil}
               \Gamma((K+1)/2)(1.1/a)^{K+1}
                  &&(u<0,\ K+u\ge2),\label{eq:RS-negative}
\end{align}
where $P_0=\sqrt2/(4\pi^2)$ and
$\kappa_0=((3-2\log2)\pi)^{-1/2}$. There is no upper bound on $K$
and no lower bound on $u$; the negative-part remainder requires
exactly the displayed condition $K+u\ge2$.

\emph{A uniform vertical Stirling estimate.}
Euler--Maclaurin gives, on the slit plane,
\begin{align*}
 \theta(z)&=\log\Gamma(z)-(z-\tfrac12)\Log z+z-\tfrac12\log(2\pi)\\
 &=\frac1{12z}-\frac1{360z^3}
   -\frac14\int_0^\infty\frac{B_4(\{u\})}{(z+u)^4}\,du.
\end{align*}
Here the branch of $\log\Gamma$ is continued from the positive real
axis. The integral is holomorphic off the nonpositive real axis, so
analytic continuation extends the identity from $\Rea z>0$.
Since $|B_4|\le1/30$ and
$\int_{-\infty}^{\infty}(u^2+v_0^2)^{-2}\,du=\pi/(2v_0^3)$,
for $v_0=|\Ima z|>1$ we obtain
\begin{equation}\label{eq:vertical-stirling}
 |\theta(z)|\le\frac1{12v_0}
       +\frac{1/360+\pi/240}{v_0^3}
       <\frac1{12(v_0-1)}.
\end{equation}
This estimate is uniform in the real part, including its unbounded
negative tail.

\emph{Elementary normalizer bounds.}
For $s=\sigma+ix/2$, $-3\le\sigma\le4$, direct differentiation gives
\begin{gather*}
 |\alpha(s)-L/2-i\pi/4|\le\frac8{x-20},\qquad
 |\alpha'(s)|\le\frac1{x-6},\\
 \Rea\alpha(w)\ge L/2,\qquad
 |\gamma|\le e^{10^{-8}y}\mathcal Q^{-y/2},\qquad
 s_A=\overline{s_B}-y+\kappa,\quad
 |\kappa|\le\frac{ty}{2(x-6)}.
\end{gather*}
The first derivative bound also holds for arbitrary real $\sigma$.
For example, integrate the real part of
$\alpha+(t/2)\alpha\alpha'$ between the two real parts to obtain the
bound for $\gamma$. All errors needed here are bounded by multiples of
$(L+1)/(x-20)$; that quotient is decreasing in this range. Also
$N\ge60000$, $\log N\le L$, and $N^{|\kappa|}<1.000001$.

\emph{Errors in the summands.}
To check the contour estimates explicitly, put $s=\sigma+iT$,
$T=x/2$, $\alpha_n=\alpha(s)-\log n$, and
$\Delta=\sqrt t V+t\alpha_n/2$ in \cite[equation (40)]{P}.
The shifted contour has imaginary part at least $T-0.08$.
Taylor's formula and \eqref{eq:vertical-stirling} bound the absolute
value of the error in the complex exponent by
\[
 \frac{tV^2/2+t^2|\alpha_n|^2/8+1/6}{T-3.08}.
\]
Indeed, $|\Delta|^2\le2tV^2+t^2|\alpha_n|^2/2$ and
$|\alpha'|\le[2(T-3.08)]^{-1}$ along the Taylor segment.
Integration against $e^{-V^2}/\sqrt\pi$ then bounds the relative
error by
\[
 \exp\left(\frac{t^2|\alpha_n|^2/8+t/4+1/6}{T-3.33}\right)-1,
\]
because $(1-t/[2(T-3.08)])^{-1/2}
\le \exp(t/[4(T-3.33)])$. This is the bound in
\cite[Proposition 6.1]{P}, now obtained with a Stirling estimate
uniform on the entire shifted contour. In our domain it is at most
$e^{d(x)}-1$, where we may take
\[
 d(x)=\frac{L^2/400+0.626}{x-6.66}<10^{-10}.
\]
Indeed, for $n\le N$,
$|\alpha-\log n|^2\le\frac14\log^2(\mathcal Q/n^2)+0.667$.
The function $d$ decreases, and $L<23$ at $x=6\cdot10^{10}$.
Combining the heat weights before estimating gives
\[
 |e^{t\log^2n/4}n^{-s_B}|\le n^{-(1+y)/2}\le n^{-2}.
\]
The reflected-to-first magnitude ratio is at most $1.01$, using
$n\le N$ and $N^2\le\mathcal Q+1/80$. Thus the sum of the absolute
main-term weights of both sums is below $4$, since
$2.01\zeta(2)<4$, and their total error
is below $10^{-8}$.

\emph{Errors in the two remainder integrals.}
Here are details of the real-part extension. They also control the
unbounded negative tail in the Gaussian integral.
Write $T=x/2$, $T'=T+\pi t/8$, $a=\sqrt{T'/(2\pi)}$, and $A=T-3$.
For $-3\le\sigma\le4$, the remainder has the form
\begin{equation}\label{eq:R}
 R_{t,N}(\sigma+iT)=\omega_Ne^{i\pi\sigma/4}
 e^{t\pi^2/64}M_0(iT')\{C_0+\e_\sigma\},
 \qquad |\omega_N|=1,
\end{equation}
where $|C_0|\le1/2$. We claim the deliberately loose bound
\begin{equation}\label{eq:eps}
 |\e_\sigma|\le
 \frac{0.3\,9^\sigma+0.04\,2^{-\sigma}}{a-1}+10^{-8}<\frac1{25}.
\end{equation}

On the shifted contour in \eqref{eq:shifted-heat}, put
$U=\sigma+\sqrt t V$, where $V$ has density $e^{-V^2}/\sqrt\pi$.
Choose the shift $\beta_N=i\pi/4$. The Stirling estimate
\eqref{eq:vertical-stirling} and Taylor expansion of $\log M_0$
about $iT'$ bound the error $E(U)$ in the complex exponent by
\[
 \frac{U^2+6|U|+2/3}{4A}\le\frac{U^2+1}{A}.
\]
Indeed $\alpha(iT')=\log a+i\pi/4+e_0$ with
$|e_0|\le3/(2T')$, and $|\alpha'|\le1/(2A)$ on the horizontal
segment. Inserting \eqref{eq:RS-source} into
\eqref{eq:shifted-heat} cancels $a^{-U}$ and the phase depending
on $V$, and gives exactly
\[
 \frac{R_{t,N}(\sigma+iT)}
      {\omega_N e^{i\pi\sigma/4}e^{t\pi^2/64}M_0(iT')}
 =\mathbb E\left[e^{E(U)}
       \left(C_0+\sum_{k=1}^{K(U)}\frac{C_k(\rho,U)}{a^k}
                    +RS_{K(U)}(U+iT')\right)\right].
\]
Thus \eqref{eq:R} is derived here, and is not an application of
Polymath's Proposition 6.3 outside its stated domain.
With $D=1-t/A$, Gaussian completion yields
\begin{align*}
 \mathbb E e^{(U^2+1)/A}-1
 &\le D^{-1/2}e^{17/A+16t/(A^2D)}-1<10^{-8},\\
 0.2\,\mathbb E(9^Ue^{(U^2+1)/A})
 &\le0.2\,9^\sigma D^{-1/2}
       e^{17/A+t(\log3+4/A)^2/D}<0.3\,9^\sigma,\\
 0.029\,\mathbb E(2^{-U}e^{(U^2+1)/A})
 &\le0.029\,2^{-\sigma}D^{-1/2}
       e^{17/A+t(\log2/2+4/A)^2/D}<0.04\,2^{-\sigma}.
\end{align*}
These bounds hold already at the worst endpoints $A=3\cdot10^{10}-3$,
$t=1/5$.

To use the coefficient bounds of \cite[Proposition 6.2]{P}, take
$K=1$ when $U\ge0$ and $K=\max\{100,\lfloor-U\rfloor+3\}$ otherwise.
For $U\ge0$ the nonconstant coefficient and remainder are bounded by
$0.2\,9^U/(a-0.865)$: from
\eqref{eq:C-positive} and \eqref{eq:RS-positive}, with $K=1$,
\[
 \frac{|C_1|}{a}+|RS_1|
 \le9^U\left(\frac{1}{2\sqrt{2\pi}\,a}
                         +\frac{1.21}{7a^2}\right)
 <9^U\left(\frac{0.2}{a}+\frac{0.173}{a^2}\right)
 \le\frac{0.2\,9^U}{a-0.865}.
\]
At $U=0$ the coefficient bound \eqref{eq:C-negative} suffices as
well. For $U<0$, our choice of $K$ gives $K+U\ge2$, so
\eqref{eq:RS-negative} applies on the entire negative tail.
The first hundred coefficients have
sum at most $0.029\,2^{-U}/a$: define
\[
 P_0=\frac{\sqrt2}{4\pi^2}<0.036,\qquad
 \kappa_0=((3-2\log2)\pi)^{-1/2}<0.445.
\]
Separating the even and odd indices gives
\[
 a\sum_{k=1}^{100}P_0\Gamma(k/2)(\kappa_0/a)^k
 \le P_0\frac{\sqrt\pi\kappa_0+\kappa_0^2/a}
                   {1-50\kappa_0^2/a^2}<0.029.
\]
When $K=100$, the additional $RS_{100}$, including the exponent-error
weight, has Gaussian expectation below $10^{-200}$. Indeed its
coefficient is at most $\tfrac12\Gamma(101/2)(1.1/a)^{101}$,
and $\mathbb E e^{(U^2+1)/A}<2$; use
$\Gamma(101/2)\le101^{101}$ and $a\ge60000$.
For all remaining indices, including the separate remainder index
$k=K+1$, one has $k\le-U+4$. Tonelli's theorem and a Gaussian tail
bound majorize their expectation by
\[
 8\sqrt{t/\pi}\,e^{33/A}
 \sum_{k\ge100}\frac{1.1^k\Gamma(k/2)}{(k-8)a^k}
 \exp\left(-\frac{(k-7)^2}{2t}+(k-7)\log2\right).
\]
Indeed, $U-\sigma\le7-k$, and $U^2+1\le2(U-\sigma)^2+33$.
Use $\Gamma(k/2)\le k^k$, $\log k\le k/4$ for $k\ge100$,
$\log1.1<0.1$, and $\log2<1$. The last display is less than
$8e^{-5000}<10^{-100}$. This proves \eqref{eq:eps}; the small
constant-mode error is below $\tfrac12\cdot10^{-8}$, so it and the
two tails fit inside the $10^{-8}$ allowance.
For its final inequality use $9^\sigma\le9^4$ and $2^{-\sigma}\le8$.

Normalize \eqref{eq:R} by $B_t$. Put $s_+=(1+y)/2+iT$,
$\sigma_+=(1+y)/2$, $\Delta=\sigma_+-i\pi t/8$, and
$L'=\log(T'/(2\pi))$. Taylor expansion from $iT'$ gives
\[
 \Rea\{\log M_0(s_+)-\log M_0(iT')\}
 =\frac{\sigma_+L'}2+\frac{t\pi^2}{32}+e_1,
 \qquad |e_1|\le\frac{22}{x-20}.
\]
Here $|\Delta|<4$, the linear error is at most $12/x$, and
the quadratic error is at most $8/(x-6)$. Also
\[
 \left|\alpha(s_+)^2-(L/2+i\pi/4)^2\right|
 \le\frac{8\sqrt{L^2+\pi^2/4}}{x-20}
                         +\frac{64}{(x-20)^2}.
\]
Using $L'\ge L$ and $\sigma_+>0$ therefore gives
\[
 \frac{e^{t\pi^2/64}|M_0(iT')|}{|M_t((1+y)/2+iT)|}
 \le2\mathcal Q^{-(1+y)/4}e^{-tL^2/16}.
\]
In this calculation the three terms involving $t\pi^2$ cancel;
a bound for the remaining error in the exponent is
$(22+2t\sqrt{L^2+\pi^2/4})/(x-20)+16t/(x-20)^2<10^{-6}$.
Both real parts $\sigma=(1\pm y)/2$ lie in $[-3,4]$.
Thus both remainder integrals together contribute less than
$4\mathcal Q^{-1}<10^{-8}$. Combining them with the summand errors
proves \eqref{eq:approx}. Dominated convergence gives the endpoint $t=0$.

Finally, if $4\le y\le6$, the first sum differs from $1$ by at most
$\zeta(5/2)-1<0.345$. The reflected sum has modulus at most
\[
 1.01N^{-y}\sum_{n\le N}n^{(y-1)/2}
 \le1.01N^{(1-y)/2}\le1.01N^{-3/2}<10^{-6}.
\]
Together with \eqref{eq:approx} this is less than $0.345002<7/20$,
proving \eqref{eq:disk}.
\end{proof}

\section{The fixed-height estimate and comparison}
\begin{proposition}\label{prop:density}
For $x\ge X$ and $0\le t\le1/5$,
\[
 S(x,4,t)>\frac{21}{4}.
\]
\end{proposition}
\begin{proof}
Fix $z_0=x+4i$ and $t$, and set
$N_0=\lfloor\sqrt{x/(4\pi)+t/16}\rfloor$. On the closed unit disk
$|z-z_0|\le1$ put
\[
 F=H_t/B_t,\qquad P=P_{N_0}(s_B,t),\qquad
 R=\gamma P_{N_0}(s_A,t),\qquad E=F-P-R.
\]
All four functions are holomorphic on a neighbourhood of the disk.
Its points satisfy $\Rea z\ge X-1>6\cdot10^{10}$ and
$3\le\Ima z\le5$. The pointwise cutoff $N(z)$ in
Lemma~\ref{lem:approx} differs from $N_0$ by at most one, since
$\sqrt{\Rea z/(4\pi)+t/16}$ varies by less than one on the disk.
The possible extra index satisfies $n\ge60000$ and $n\le N(z)+1$.
The normalizer bounds in that lemma give a reflected-to-first ratio
below $1.02$ also at $n=N(z)+1$; use
$n\le\sqrt{\mathcal Q+1/80}+1$ and $3\le y\le5$.
Both heat weights still combine as before because $\log n\le L$.
Consequently
\begin{equation}\label{eq:frozen-error}
 |E(z)|<10^{-6}+\frac{3}{60000^2}<2\cdot10^{-6},
 \qquad |E'(z_0)|<2\cdot10^{-6}.
\end{equation}
The second inequality is Cauchy's estimate on the fixed-cutoff
remainder. At the centre, $N_0=N(z_0)$ and $|E(z_0)|<10^{-6}$.

Hereafter all quantities are evaluated at $z_0$. The elementary bounds
proved in Section~\ref{sec:checks} are
\[
 Z:=\sum_{n\ge1}n^{-5/2}<\frac{269}{200},\qquad
 J:=\sum_{n\ge2}(\log n)n^{-5/2}<\frac{39}{100}.
\]
Since $\Rea w=5/2$ and
$|s_B'|\le1/2+t/[4(x-6)]<1/2+10^{-10}$, the heat-weight estimate gives
\[
 |P-1|\le Z-1,\qquad |P'|\le(1/2+10^{-10})J.
\]
For the reflected term, the same estimates give
\[
 |R|\le1.01N_0^{-4}\sum_{n\le N_0}n^{3/2}
       \le1.01N_0^{-3/2}<10^{-6}.
\]
Write $\beta=\alpha+(t/2)\alpha\alpha'$. Then
$\gamma'/\gamma=(i/2)(\beta(v)+\beta(w))$ and
$|s_A'|\le1/2+t/[4(x-6)]$. Thus
\[
 |\gamma'/\gamma|+|s_A'|\log n\le L+2
       \qquad(1\le n\le N_0),
\]
using $|\alpha(s)|\le L/2+1$ and the preceding derivative bounds.
It follows, term by term, that
\[
 |R'|\le1.01(L+2)N_0^{-3/2}<2\cdot10^{-6}.
\]
To check uniformity in $x$, use $\mathcal Q<(N_0+1)^2$ and hence
$L+2<2\log N_0+3$. The function
$(2\log N+3)N^{-3/2}$ decreases for $N\ge60000$; with
$\log60000<12$ and $\sqrt{60000}>240$, its contribution is at most
$1.01\cdot27/(60000\cdot240)<2\cdot10^{-6}$.

Combining these inequalities yields
\begin{equation}\label{eq:F-logder}
 \left|\frac{F'}F\right|
 <\frac{(1/2+10^{-10})(39/100)+4\cdot10^{-6}}
          {2-269/200-2\cdot10^{-6}}
 <0.297717.
\end{equation}
Finally
\[
 -\Ima\frac{B_t'}{B_t}
 =\frac12\Rea\left(\alpha(w)+\frac t2\alpha(w)\alpha'(w)\right)
 \ge\frac14\log\frac{x}{4\pi}-10^{-6}.
\]
Section~\ref{sec:checks} gives $L/4>111/20$. Therefore
\[
 S(x,4,t)>\frac{111}{20}-10^{-6}
 -\frac{(1/2+10^{-10})(39/100)+4\cdot10^{-6}}
          {2-269/200-2\cdot10^{-6}}
 >\frac{21}{4}.
\]
This estimate includes the heat dependence for the full interval
$0\le t\le1/5$; it is not a calculation restricted to $t=0$.
\end{proof}

For reference, the preceding estimate gives the one-step bound
\begin{equation}\label{eq:old-bound}
 \Lambda\le\frac1{50}+\frac{988}{4675}\log\frac{73691}{34608}
 =0.1797275414\ldots
\end{equation}
under the seven-eighths hypothesis. We give the comparison argument
before its stepped refinement.
\begin{proof}[Derivation of \eqref{eq:old-bound}]
The functional equation converts the assumed half-plane to
$M(0)\le a=9/16$. At a simple highest zero $x+ih$, with $q=h^2$,
the paired-zero identity gives
\[
 q'\le-2-\frac{4q}{p}S(x,p,t)+\frac{8q}{p^2-q}
 \qquad(p^2\ge5q).
\]
For completeness, the kernel comparison behind this identity follows
by clearing positive denominators: with $D=d^2$, $A=h^2$, $B=p^2$,
$r=v^2\le A$, its numerator is a positive factor times
\[
 (B-A)\{D^2+(A+B-6r)D+(A-r)(B-r)\}\ge0.
\]
While the comparison height satisfies $q>b$, Lemma~\ref{lem:head}
places a contact zero outside the finite region, so symmetry permits
$x>X$. Proposition~\ref{prop:density}, with $p=4$, then gives
\[
 q'\le-2-\frac{21}{4}q+\frac{8q}{16-q}\le-2-kq,
 \qquad k=\frac{21}{4}-\frac8{16-a}=\frac{4675}{988}.
\]
Here the classical strip theorem ensures $q\le a$ throughout.
Use the barrier
\[
 Q(t)=(a+2/k)e^{-kt}-2/k.
\]
For $0<\e<1$ let $Q_\e'=-2-kQ_\e+\e$ and
$Q_\e(0)=a+\e$. Initially $Q_\e>a\ge M$, so contact is impossible
on some interval $[0,t_0]$. At a first simple contact while $Q_\e>b$,
every zero has height at most $\sqrt{Q_\e}$ and the classical strip
bound still gives $q\le a$. The kernel comparison therefore applies
and gives $q'\le-2-kq<Q_\e'$, contradicting contact from below.
For a multiple contact, the vertical splitting described in
Lemma~\ref{lem:head} supplies an earlier zero higher by $c\sqrt v$,
$c>0$, whereas $\sqrt{Q_\e(t-v)}=\sqrt{Q_\e(t)}+O(v)$.
This also contradicts first contact. Polymath's Theorem 1.5(i) puts
every nonreal zero at $|x|<\exp(C/t)$ for an absolute $C$, uniformly
on $t\in[t_0,1/5]$. Thus contact is attained in a compact set.
Continuous dependence of $Q_\e$ on $\e$ allows passage to
$\e\downarrow0$ up to the time $Q=b$.

The time at which $Q=b$ is
\[
 t_* =\frac1k\log\frac{2+ka}{2+kb}
      =\frac{988}{4675}\log\frac{73691}{34608}<\frac4{25}<\frac15.
\]
Comparison up to this time and continuity give $M(t_*)\le1/25$.
De Bruijn's strip theorem yields $\Lambda\le t_*+1/50<9/50$.
\end{proof}

\begin{remark}[The half-plane parameter and its admissible range]\label{rem:halfplane}
Assume more generally that the nontrivial zeros of $\zeta$ have
$\Rea s\le\sigma$, with $1/2\le\sigma\le1$. Set
\[
 a=(2\sigma-1)^2,\qquad b=1/25,\qquad
 k_4(a)=\frac{21}{4}-\frac8{16-a}.
\]
If $a\le b$, the classical strip theorem gives $\Lambda\le a/2$.
For $a>b$, define
\[
 T_4(a)=\frac1{k_4(a)}\log\frac{2+k_4(a)a}{2+k_4(a)b},
 \qquad F_4(\sigma)=\frac1{50}+T_4(a).
\]
The preceding proof gives $\Lambda\le F_4(\sigma)$ whenever
$T_4(a)\le1/5$. This time restriction is essential. Also
$4^2\ge5a$ throughout $0\le a\le1$, so the kernel condition
causes no further restriction.

The function $T_4$ increases strictly with $a>b$: its numerator
interval expands while $k_4(a)$ decreases, as is also clear from
$T_4(a)=\int_b^a(2+k_4(a)q)^{-1}\,dq$.
The unique thresholds satisfy
\[
 \begin{array}{lll}
 F_4(\sigma)=1/5:&0.90664<\sigma<0.90665,
       &\sigma\simeq0.9066425127,\\
 T_4((2\sigma-1)^2)=1/5:&0.93835<\sigma<0.93836,
       &\sigma\simeq0.9383580289.
 \end{array}
\]
Thus this formula improves the published $0.2$ below the first
threshold and remains justified up to the second. Beyond the
admissible range use the unconditional bound, not an extrapolation
of $F_4$. Where both apply one may take
$\Lambda\le\min\{1/5,F_4(\sigma)\}$.
\end{remark}

\section{Heat gain and an Euler-product estimate at height 301/100}
In this section fix
\[
 p=\frac{301}{100},\qquad r=\frac1{100},\qquad
 \sigma_p=\frac{1+p}{2}=\frac{401}{200},\qquad
 h_n=e^{t\log^2n/4}.
\]
The elementary exponential check in Section~\ref{sec:step-checks}
gives $L>22.28$ for $x\ge X$, hence $L/4>557/100$.

\begin{lemma}[Heat-weight gain]\label{lem:heat-gain}
For $x\ge X$, $0\le t\le1/5$, $y\ge3$, and $1\le n\le N$,
\[
 |h_nn^{-s_B}|\le n^{-(1+y)/2}
    \exp\left(-\frac t4\log n\,(L-\log n)\right)
 \le n^{-(1+y)/2-2.78t},
 \qquad |n^{-s_B}|\le n^{-(1+y)/2-5.57t}.
\]
\end{lemma}
\begin{proof}
Use $\Rea s_B\ge(1+y)/2+tL/4$ from the normalizer bounds.
Also $N^2\le\mathcal Q+1/80$, so
$2\log n\le L+3\cdot10^{-12}$ and
$(L-\log n)/4\ge L/8-10^{-12}>2.78$.
\end{proof}

Fix $z_0=x+ip$, $x\ge X$, and fix $t\in[0,1/5]$. On
$|z-z_0|\le r$ use the fixed cutoff $N_0=N(z_0)$ and put
\[
 F=H_t/B_t,\qquad P=P_{N_0}(s_B,t),\qquad
 R=\gamma P_{N_0}(s_A,t),\qquad E=F-P-R.
\]
The disk lies in $3\le\Ima z\le3.02$, $\Rea z>6\cdot10^{10}$.
The cutoff changes by at most one across it. As in
\eqref{eq:frozen-error}, the possible extra summands have total
modulus below $3/60000^2$: $\log n\le L$ still holds for
$n\le N(z)+1$, so each first-sum weight is at most $n^{-2}$,
and the reflected-to-first ratio is below $1.02$. Therefore
\[
 |E|<2\cdot10^{-6}\text{ on the disk},\qquad
 |E'(z_0)|<2\cdot10^{-4},\qquad |E(z_0)|<10^{-6}.
\]
At the centre, the normalizer estimates give, for $n\le N_0$,
\[
 |\gamma h_nn^{-s_A}|\le1.01N_0^{-p}n^{(p-1)/2},\qquad
 |\gamma'/\gamma|+|s_A'|\log n\le L+2.
\]
Consequently, using $L+2<2\log N_0+3$ and $N_0\ge60000$,
\begin{align*}
 |R|&\le1.01N_0^{(1-p)/2}<1.7\cdot10^{-5},\\
 |R'|&\le1.01(2\log N_0+3)N_0^{(1-p)/2}
       \le1.01\frac{2\log N_0+3}{N_0}<4.6\cdot10^{-4}.
\end{align*}
For the last inequality, $(2\log N+3)/N$ decreases for $N\ge60000$
and $\log60000<12$. We may thus use the uniform constants
\begin{equation}\label{eq:step-errors}
 \epsilon_0=2\cdot10^{-5}>|R|+|E(z_0)|,\qquad
 \epsilon_1=6.6\cdot10^{-4}>|R'|+|E'(z_0)|,\qquad
 \nu=\tfrac12+10^{-10}>|s_B'|.
\end{equation}
In addition, the normalizer has
$-\Ima(B_t'/B_t)\ge L/4-10^{-6}>557/100-10^{-6}$.

For $u>1$, write
\[
 Z(u)=\zeta(u),\qquad
 J(u)=-\zeta'(u)=\sum_{n\ge2}(\log n)n^{-u},\qquad
 \psi(u)=-\frac{\zeta'}{\zeta}(u)
       =\sum_{n\ge2}\Lambda_{\rm vM}(n)n^{-u},
\]
where $\Lambda_{\rm vM}$ denotes the von Mangoldt function.
These three functions are decreasing.

\begin{proposition}[Triangle estimate]\label{prop:triangle-step}
For $x\ge X$ and $0\le t^-\le t\le1/5$, put
$u=\sigma_p+2.78t^-$. Then
\begin{equation}\label{eq:triangle-step}
 S(x,p,t)>\frac{557}{100}-10^{-6}
             -\frac{\nu J(u)+\epsilon_1}{2-Z(u)-\epsilon_0}.
\end{equation}
\end{proposition}
\begin{proof}
Lemma~\ref{lem:heat-gain} gives $|P-1|\le Z(u)-1$ and
$|P'|\le\nu J(u)$. The denominator is positive because
$Z(u)\le Z(\sigma_p)<1.641$, as certified below. Thus
$|F'/F|\le(\nu J(u)+\epsilon_1)/(2-Z(u)-\epsilon_0)$.
Combine this with the normalizer bound.
\end{proof}

\begin{proposition}[Euler-product perturbation estimate]\label{prop:euler-step}
Let $0\le t^-\le t\le t^+\le1/5$ and $x\ge X$. Define
$u=\sigma_p+5.57t^-$, $m=1/Z(u)$, $K_0=65536$, and
\begin{align}
 \delta_0&=\sum_{n=2}^{K_0}
    (e^{t^+\log^2n/4}-1)n^{-u}+4\cdot10^{-5},\label{eq:delta0}\\
 \delta_1&=\sum_{n=2}^{K_0}
    (e^{t^+\log^2n/4}-1)(\log n)n^{-u}+5\cdot10^{-4}.
       \label{eq:delta1}
\end{align}
If $m>\delta_0+\epsilon_0$, then
\begin{equation}\label{eq:euler-step}
 S(x,p,t)>\frac{557}{100}-10^{-6}
    -\frac{\nu(\psi(u)m+\delta_1)+\epsilon_1}
                 {m-\delta_0-\epsilon_0}.
\end{equation}
\end{proposition}
\begin{proof}
Set $s=s_B(z_0)$, with $\Rea s\ge\sigma_p+5.57t\ge u$, and
keep $N_0,t$ fixed when differentiating in $z$. Then
\begin{align*}
 P&=\zeta(s)-\sum_{n>N_0}n^{-s}+D_0,
 &D_0&=\sum_{n\le N_0}(h_n-1)n^{-s},\\
 P'&=-s_B'\left(-\zeta'(s)-\sum_{n>N_0}(\log n)n^{-s}+D_1\right),
 &D_1&=\sum_{n\le N_0}(h_n-1)(\log n)n^{-s}.
\end{align*}
The integral test, $\Rea s\ge\sigma_p>2$, and $N_0\ge60000$ give
\[
 \sum_{n>N_0}n^{-\Rea s}<1.7\cdot10^{-5},\qquad
 \sum_{n>N_0}(\log n)n^{-\Rea s}<2.1\cdot10^{-4}.
\]
In $D_j$, the terms with $n\le K_0$ are bounded by the finite sums
in \eqref{eq:delta0}--\eqref{eq:delta1}. For $K_0<n\le N_0$,
\[
 |(h_n-1)n^{-s}|\le|h_nn^{-s}|\le n^{-\sigma_p}.
\]
The corresponding tails above $K_0$ are below $1.6\cdot10^{-5}$
and $1.9\cdot10^{-4}$, respectively. The allowances in the
proposition therefore cover both the zeta tails and the perturbation
tails, whether $N_0\le K_0$ or $N_0>K_0$.

The absolutely convergent Euler product in $\Rea s>1$ gives
\[
 \left|\frac{\zeta'(s)}{\zeta(s)}\right|\le\psi(u),\qquad
 |1/\zeta(s)|\le Z(u),\qquad \rho:=|\zeta(s)|\ge m.
\]
Hence
\[
 \left|\frac{F'}F\right|
 \le\frac{\nu(\psi(u)\rho+\delta_1)+\epsilon_1}
                 {\rho-\delta_0-\epsilon_0}.
\]
This quotient decreases in $\rho>\delta_0+\epsilon_0$, so replacing
$\rho$ by $m$ gives \eqref{eq:euler-step}. The Euler product is used
only for $\zeta$; the heat weights themselves are not multiplicative.
\end{proof}

\section{Stepped strip comparison}\label{sec:stepped}
The next statement applies to an arbitrary initial strip, so it also
gives a finite procedure for converting a half-plane hypothesis
$\Rea\rho\le\sigma$ into a bound: start with $q_0=(2\sigma-1)^2$.

\begin{theorem}\label{thm:step}
Suppose $M(0)\le q_0$ and
$0=t_0<\cdots<t_n\le1/5$. Write $d_i=t_{i+1}-t_i$.
Let $q_0>q_1>\cdots>q_n\ge b=1/25$ and fix $p>0$.
Suppose, for $0\le i<n$, that
\[
 S(x,p,t)\ge C_i\quad(x\ge X,\ t_i\le t\le t_{i+1}),\qquad
 p^2\ge5q_i,\qquad
 k_i=\frac{4C_i}{p}-\frac8{p^2-q_i}>0,
\]
and
\begin{equation}\label{eq:step-gate}
 e^{k_id_i}(2+k_iq_{i+1})\ge2+k_iq_i.
\end{equation}
Then $M(t_n)\le q_n$ and $\Lambda\le t_n+q_n/2$.
\end{theorem}
\begin{proof}
Induct on $i$. Given $M(t_i)\le q_i$, the classical strip theorem
ensures $M(t)\le q_i$ for $t\ge t_i$. At a highest simple zero
with squared height $q>b$, Lemma~\ref{lem:head} places the zero at
$|x|>X$, and the paired-zero comparison used above gives
\[
 q'\le-2-\frac{4C_i}{p}q+\frac{8q}{p^2-q}
       \le-2-k_iq.
\]
Compare with $Q_i(t)=(q_i+2/k_i)e^{-k_i(t-t_i)}-2/k_i$ while
$Q_i>b$. The strict perturbed barriers, compact attainment at positive
time, and backward Hermite splitting used in the proof of
\eqref{eq:old-bound} justify this comparison also at multiple zeros.
At the initial time of a step, a small upward perturbation supplies
strict separation; the strip theorem bounds contact heights by $q_i$.
If $Q_i$ reaches $b$ before $t_{i+1}$, then $M\le b$ there and
remains at most $b$ by the strip theorem. Otherwise
\eqref{eq:step-gate} is exactly $Q_i(t_{i+1})\le q_{i+1}$.
This proves the induction. Apply the strip theorem once more at $t_n$.
\end{proof}

Both tables below use $p=301/100$, $q_0=9/16$, and $q_n=1/25$.
The constants $C_i$ in Table~\ref{tab:tri-steps} follow from
Proposition~\ref{prop:triangle-step}, with $t^-=t_i$; those in
Table~\ref{tab:eul-steps} follow from Proposition~\ref{prop:euler-step},
with $t^-=t_i$, $t^+=t_{i+1}$. Section~\ref{sec:step-checks}
supplies finite upper bounds for every scalar quantity involved.

\begin{table}[htbp]
\centering\small
\caption{Triangle estimate: final time $3793/25000$, giving $\Lambda\le0.17172$.}\label{tab:tri-steps}
\begin{tabular}{rccccc}
\toprule
$i$ & $t_i$ & $10^6q_i$ & $10^4C_i$ & $d_i$ & $10^6q_{i+1}$\\
\midrule
0 & $0$ & 562500 & 42786 & $\frac{1}{100}$ & 516903\\
1 & $\frac{1}{100}$ & 516903 & 44307 & $\frac{1}{100}$ & 472419\\
2 & $\frac{1}{50}$ & 472419 & 45557 & $\frac{1}{100}$ & 429333\\
3 & $\frac{3}{100}$ & 429333 & 46600 & $\frac{1}{100}$ & 387828\\
4 & $\frac{1}{25}$ & 387828 & 47483 & $\frac{1}{100}$ & 348016\\
5 & $\frac{1}{20}$ & 348016 & 48237 & $\frac{1}{100}$ & 309958\\
6 & $\frac{3}{50}$ & 309958 & 48888 & $\frac{1}{100}$ & 273677\\
7 & $\frac{7}{100}$ & 273677 & 49455 & $\frac{1}{100}$ & 239169\\
8 & $\frac{2}{25}$ & 239169 & 49952 & $\frac{1}{100}$ & 206410\\
9 & $\frac{9}{100}$ & 206410 & 50391 & $\frac{1}{100}$ & 175361\\
10 & $\frac{1}{10}$ & 175361 & 50780 & $\frac{1}{100}$ & 145974\\
11 & $\frac{11}{100}$ & 145974 & 51127 & $\frac{1}{100}$ & 118194\\
12 & $\frac{3}{25}$ & 118194 & 51439 & $\frac{1}{100}$ & 91960\\
13 & $\frac{13}{100}$ & 91960 & 51720 & $\frac{1}{100}$ & 67208\\
14 & $\frac{7}{50}$ & 67208 & 51974 & $\frac{1}{100}$ & 43874\\
15 & $\frac{3}{20}$ & 43874 & 52204 & $\frac{43}{25000}$ & 40000\\
\bottomrule
\end{tabular}
\end{table}

\begin{table}[htbp]
\centering\small
\caption{Euler-product estimate: final time $1801/12500$, giving $\Lambda\le0.16408$.}\label{tab:eul-steps}
\begin{tabular}{rccccc}
\toprule
$i$ & $t_i$ & $10^6q_i$ & $10^4C_i$ & $d_i$ & $10^6q_{i+1}$\\
\midrule
0 & $0$ & 562500 & 52708 & $\frac{1}{100}$ & 510004\\
1 & $\frac{1}{100}$ & 510004 & 52851 & $\frac{1}{100}$ & 460480\\
2 & $\frac{1}{50}$ & 460480 & 53006 & $\frac{1}{100}$ & 413770\\
3 & $\frac{3}{100}$ & 413770 & 53166 & $\frac{1}{100}$ & 369731\\
4 & $\frac{1}{25}$ & 369731 & 53323 & $\frac{1}{100}$ & 328228\\
5 & $\frac{1}{20}$ & 328228 & 53475 & $\frac{1}{100}$ & 289133\\
6 & $\frac{3}{50}$ & 289133 & 53619 & $\frac{1}{100}$ & 252324\\
7 & $\frac{7}{100}$ & 252324 & 53754 & $\frac{1}{100}$ & 217683\\
8 & $\frac{2}{25}$ & 217683 & 53881 & $\frac{1}{100}$ & 185096\\
9 & $\frac{9}{100}$ & 185096 & 53999 & $\frac{1}{100}$ & 154454\\
10 & $\frac{1}{10}$ & 154454 & 54108 & $\frac{1}{100}$ & 125652\\
11 & $\frac{11}{100}$ & 125652 & 54209 & $\frac{1}{100}$ & 98588\\
12 & $\frac{3}{25}$ & 98588 & 54303 & $\frac{1}{100}$ & 73166\\
13 & $\frac{13}{100}$ & 73166 & 54390 & $\frac{1}{100}$ & 49294\\
14 & $\frac{7}{50}$ & 49294 & 54470 & $\frac{51}{12500}$ & 40000\\
\bottomrule
\end{tabular}
\end{table}

\begin{proof}[Derivation of the two intermediate stepped bounds]
The half-plane assumption and the functional equation give
$M(0)\le9/16$. Every row of Table~\ref{tab:eul-steps} satisfies the
hypotheses of Theorem~\ref{thm:step}, by the certificates below.
The final time is $t_{15}=7/50+51/12500=1801/12500$, and
$q_{15}=1/25$. Therefore
\[
 \Lambda\le\frac{1801}{12500}+\frac1{50}
          =\frac{2051}{12500}=0.16408.
\]
The same reasoning applied to Table~\ref{tab:tri-steps} gives
$\Lambda\le3793/25000+1/50=4293/25000=0.17172$.
\end{proof}

\section{A kernel comparison using the vertical derivative}\label{sec:kernel}
The preceding argument uses only a lower bound for $S$. An upper
bound for $\partial_pS$ permits a stronger pairwise inequality.
In this section, $K_Y(d,v)$ denotes the conjugate-pair kernel in the
proof of Lemma~\ref{lem:head}.

\begin{lemma}\label{lem:kernel-derivative}
Let $p>0$, $B=p^2$, $0<A_2<B$, and $\lambda,\mu\ge0$. For
$D\ge0$ and $0\le r\le A_2$ put
\begin{align}
 w&=D+B-r,\qquad \Delta=w^2+4Dr,\notag\\
 \mathcal N&=2\lambda w\Delta-\mu(2w+4B)\Delta+8\mu Bw^2,\notag\\
 \Phi(D,r)&=2\Delta^2-(D+A_2+3r)\mathcal N.\label{eq:kernel-poly}
\end{align}
If $\Phi(D,r)\ge0$ throughout that domain, then, for
$0<h^2\le A_2$, $0\le v\le h$, and $(d,v)\ne(0,h)$,
\begin{equation}\label{eq:kernel-derivative}
 K_h(d,v)\ge h\left(\frac\lambda p K_p(d,v)
                              -\mu\partial_pK_p(d,v)\right).
\end{equation}
\end{lemma}
\begin{proof}
Set $D=d^2$, $r=v^2$. Direct differentiation gives
\[
 K_p=\frac{2pw}{\Delta},\qquad
 \partial_pK_p=\frac{(2w+4B)\Delta-8Bw^2}{\Delta^2}.
\]
Thus the bracket in \eqref{eq:kernel-derivative} is
$\mathcal N/\Delta^2$. Both $w$ and $\Delta$ are positive.
With $A=h^2$ and $v_0=D+A-r$, we have $v_0\ge D$, $v_0>0$, and
\[
 \frac{K_h}{h}=\frac{2v_0}{v_0^2+4Dr}
 \ge\frac2{v_0+4r}=\frac2{D+A+3r}
 \ge\frac2{D+A_2+3r}.
\]
If $\mathcal N\le0$, the result follows immediately. Otherwise,
$\Phi\ge0$ gives $2/(D+A_2+3r)\ge\mathcal N/\Delta^2$.
\end{proof}

For $\lambda=1$, $\mu=0$, the polynomial becomes
$2\Delta\{w(B-A_2-4r)+4Dr\}$, which is nonnegative if $B\ge5A_2$.
Thus the lemma includes the earlier comparison. The new table uses
$1.718\le\lambda\le1.920$, with the derivative term paying for the
larger coefficient of $K_p$.

\subsection{Two density estimates at height \texorpdfstring{$61/20$}{61/20}}\label{sec:two-density}
In this subsection fix
\[
 p=\frac{61}{20},\quad r_0=\frac1{20},\quad
 \sigma_p=\frac{1+p}{2}=\frac{81}{40},\quad
 N_*=69098,\quad N_1=65536.
\]
Section~\ref{sec:kernel-checks} certifies $L/4>5.5716$ and $N\ge N_*$
for $x\ge X$. The normalizer formulas and
Lemma~\ref{lem:heat-gain} also give, for $n\le N$ and $3\le y\le6$,
\begin{equation}\label{eq:reflected-heat}
 |\gamma h_nn^{-s_A}|\le1.01(n/N)^y n^{-(1+y)/2-2.78t}.
\end{equation}
Indeed, use $s_A=\overline{s_B}-y+\kappa$,
$|\gamma|\le e^{10^{-8}y}\mathcal Q^{-y/2}$, and
$N^{|\kappa|}<1.000001$; replacing $\mathcal Q^{1/2}$ by $N$
is absorbed in the factor $1.01$.

For $0\le t^-\le t^+\le1/5$, put
\[
 u=\sigma_p+5.57t^-,\qquad u_8=\sigma_p+2.78t^-,\qquad
 e=\frac{p-1}{2}-2.78t^-.
\]
In particular $0.469\le e\le1.025$. Define, for $j=0,1,2$,
\begin{align}
 \delta_j&=\sum_{n=2}^{N_1}(e^{t^+\log^2n/4}-1)
                  (\log n)^j n^{-u}
        +\int_{N_1}^\infty(\log y)^j y^{-u_8}\,dy,\label{eq:new-delta}\\
 \rho_0&=1.01N_*^{-p}
          \left(\frac{N_*^{e+1}}{e+1}+N_*^e\right),\notag\\
 L_*&=2\log N_*+3,\qquad
 \rho_1=L_*\rho_0,\qquad \rho_2=(L_*^2+10^{-9})\rho_0,\label{eq:new-rho}\\
 \eta_0&=10^{-6},\qquad \eta_1=4\cdot10^{-5},\qquad
 \eta_2=16\cdot10^{-4}.\notag
\end{align}
Write $Z=\zeta(u)$, $\psi=-\zeta'(u)/\zeta(u)$, and
\[
 \psi_2=\left(\frac{\zeta'}{\zeta}\right)'(u)
       =\sum_{n\ge2}\Lambda_{\rm vM}(n)(\log n)n^{-u},\qquad m=1/Z.
\]
All these quantities are positive. Set
$c_1=1/2+10^{-10}$, $c_2=1/4+10^{-10}$, and
\begin{align}
 e_0&=\delta_0+\rho_0+\eta_0,
 &\varepsilon_0&=Ze_0,\notag\\
 e_1&=c_1\delta_1+\rho_1+\eta_1,
 &\varepsilon_1&=Z(e_1+c_1\psi e_0),\notag\\
 e_2&=c_2\delta_2+\rho_2+\eta_2+10^{-9},
 &\varepsilon_2&=Z\{e_2+2c_1\psi e_1+c_2(\psi_2+\psi^2)e_0\}.
 \label{eq:new-errors}
\end{align}

\begin{proposition}\label{prop:two-density}
If $\varepsilon_0<1/2$, then for $x\ge X$ and $t^-\le t\le t^+$,
\begin{align}
 S(x,p,t)&\ge C^*:=5.5716-10^{-6}-c_1\psi
                         -\frac{\varepsilon_1}{1-\varepsilon_0},
                         \label{eq:new-C}\\
 \partial_pS(x,p,t)&\le G^*:=c_2\psi_2
       +\frac{\varepsilon_2}{1-\varepsilon_0}
       +\left(\frac{\varepsilon_1}{1-\varepsilon_0}\right)^2
       +2\cdot10^{-10}.\label{eq:new-G}
\end{align}
\end{proposition}
\begin{proof}
All primes below denote complex differentiation in $z$, with $t$ fixed.
Write $H_t=B_tF$. Since $S=-\Ima(H_t'/H_t)$,
$\partial_yS=-\Rea(H_t'/H_t)'$.

\emph{Normalizer.} Direct differentiation gives
\[
 \alpha''(s)=s^{-3}+2(s-1)^{-3}-\tfrac12s^{-2},\qquad
 |\alpha''|\le2/x^2+24/x^3<3/x^2.
\]
With $\beta=\alpha+(t/2)\alpha\alpha'$, one has
$B_t'/B_t=-i\beta(w)/2$ and $(B_t'/B_t)'=-\beta'(w)/4$.
The bounds already proved for $\alpha$ and $\alpha'$ yield
\[
 |\beta'|\le|\alpha'|+\tfrac t2(|\alpha'|^2+|\alpha||\alpha''|)
 <10^{-10},\qquad |s_B''|\le t|\alpha''|/8<10^{-20}.
\]
More precisely, $|s_B'|\le1/2+t/[4(x-6)]<1/2+10^{-12}$.
Consequently $|s_B'|\le c_1$ and $|s_B'|^2\le c_2$; the latter
inequality is obtained from this sharper bound, not by squaring $c_1$.

\emph{Fixed cutoff.} At $z_0=x+ip$ put $N_0=N(z_0)$ and, on
$|z-z_0|\le r_0$, define
\[
 P=P_{N_0}(s_B,t),\quad R=\gamma P_{N_0}(s_A,t),\quad E=F-P-R,
 \qquad T=F-\zeta(s_B),\quad \varepsilon=T/\zeta(s_B).
\]
The disk lies in $3\le\Ima z\le3.1$ and $\Rea z>6\cdot10^{10}$.
The cutoff changes by at most one. The same argument as in
\eqref{eq:frozen-error}, with extra summands bounded by
$3/60000^2$, gives $|E|<2\cdot10^{-6}$ on the disk and
$|E(z_0)|<10^{-6}$. Cauchy's estimates therefore give
$|E'(z_0)|\le\eta_1$ and $|E''(z_0)|\le\eta_2$.

At the centre, $\Rea s_B\ge u$. The finite perturbation
$P-\zeta(s_B)$ has terms $(h_n-1)n^{-s_B}$ for $n\le N_0$
and $-n^{-s_B}$ for $n>N_0$.
Because $N_0\ge N_*>N_1$, the perturbation tail $N_1<n\le N_0$
and the zeta tail $n>N_0$ are disjoint. Each is bounded by
$n^{-u_8}$, so their union is covered by the single integral tail
in \eqref{eq:new-delta}; there is no omitted zeta tail.
Termwise differentiation then gives
\begin{align*}
 |P-\zeta(s_B)|&\le\delta_0,\\
 |(P-\zeta(s_B))'|&\le c_1\delta_1,\\
 |(P-\zeta(s_B))''|&\le c_2\delta_2+10^{-20}\delta_1.
\end{align*}
Here the integral test applies since $(\log y)^j y^{-u_8}$ decreases
for $y\ge N_1$, $j\le2$.

For $R$, equation~\eqref{eq:reflected-heat} and $e>0$ give
\[
 \sum_{n\le N_0}|\gamma h_nn^{-s_A}|
 \le1.01N_0^{-p}\left(\frac{N_0^{e+1}}{e+1}+N_0^e\right).
\]
A reflected summand $\tau_n$ has $\tau_n'=\lambda_n\tau_n$,
$\tau_n''=(\lambda_n^2+\lambda_n')\tau_n$, where
$\lambda_n=\gamma'/\gamma-s_A'\log n$.
The normalizer bounds give $|\lambda_n|\le L+2$ and
$|\lambda_n'|<10^{-9}$: use
$\gamma'/\gamma=(i/2)(\beta(v)+\beta(w))$ and the bounds for
$\beta'$ and $s_A''$.
Since $L+2<2\log N_0+3$, the preceding sum multiplied by
$(2\log N_0+3)^j$, $j\le2$, decreases with $N_0\ge N_*$.
Indeed, each power of $N_0$ has exponent at most $1-u_8\le-1.025$,
whereas its logarithmic factor contributes at most
$4/(2\log N_0+3)<0.16$ to its logarithmic derivative.
Thus $|R^{(j)}|\le\rho_j$, for $j=0,1,2$.
It follows that $|T|\le e_0$, $|T'|\le e_1$, and
\[
 |T''|\le c_2\delta_2+\rho_2+\eta_2+10^{-20}\delta_1.
\]

\emph{Division by zeta.} Put $\varrho=\zeta'/\zeta$ and
$\iota=1/\zeta(s_B)$. The absolutely convergent Euler product gives
$|\varrho(s_B)|\le\psi$, $|\varrho'(s_B)|\le\psi_2$, and
$|\iota|\le Z$. Differentiation yields
\[
 |\iota'|\le Zc_1\psi,\qquad
 |\iota''|\le Z\{10^{-20}\psi+c_2(\psi_2+\psi^2)\}.
\]
The small $s_B''$ term in this second inequality must be retained.
In the estimate for $(T\iota)''$, the two such contributions total
at most $10^{-20}Z(\delta_1+\psi e_0)$.
Since $Ze_0<1/2$, we have $\delta_0<1/2$ and $e_0<1/2$.
Moreover $\psi\le J(2)<1$ and $\delta_1<7$: the finite part is
at most $(\log N_1)\delta_0<6$ and its integral tail is below
$10^{-3}$. These elementary inequalities are included in
Section~\ref{sec:kernel-checks}. Therefore the contribution is
less than $8\cdot10^{-20}Z$, covered by the allowance $10^{-9}Z$
in $e_2$. Leibniz's rule now proves
$|\varepsilon^{(j)}|\le\varepsilon_j$, $j=0,1,2$.

Finally $F=\zeta(s_B)(1+\varepsilon)$, so
\begin{align*}
 F'/F&=s_B'\varrho(s_B)+\varepsilon'/(1+\varepsilon),\\
 (F'/F)'&=s_B''\varrho(s_B)+(s_B')^2\varrho'(s_B)
       +\varepsilon''/(1+\varepsilon)
       -(\varepsilon'/(1+\varepsilon))^2.
\end{align*}
Combine these identities with
$-\Ima(B_t'/B_t)\ge L/4-10^{-6}$ and
$|(B_t'/B_t)'|<10^{-10}$ to obtain \eqref{eq:new-C}--\eqref{eq:new-G}.
In particular, the last $2\cdot10^{-10}$ covers both the normalizer
and $s_B''\varrho$ terms. The time endpoints in the proposition
were inserted termwise in the heat bounds and the positive finite
sums; no monotonicity of $H_t'/H_t$ itself is asserted.
\end{proof}

\subsection{The stronger stepped theorem}
\begin{theorem}\label{thm:derivative-step}
Assume $M(0)\le q_0<B=p^2$ and choose
$0=t_0<\cdots<t_n\le1/5$, $d_i=t_{i+1}-t_i$, and
$q_0>q_1>\cdots>q_n\ge b=1/25$.
For $0\le i<n$, suppose $C_i,G_i,\lambda_i,\mu_i\ge0$ satisfy
\[
 S(x,p,t)\ge C_i,\qquad \partial_pS(x,p,t)\le G_i
 \quad(x\ge X,\ t_i\le t\le t_{i+1}),
\]
and the polynomial \eqref{eq:kernel-poly}, with
$(A_2,\lambda,\mu)=(q_i,\lambda_i,\mu_i)$, is nonnegative for
$D\ge0$, $0\le r\le q_i$. Define
\begin{equation}\label{eq:new-k}
 k_i=4\left\{\frac{\lambda_i C_i}{p}-\mu_iG_i
       -\frac{2\lambda_i}{B-q_i}
       -\frac{2\mu_i(B+q_i)}{(B-q_i)^2}\right\}>0.
\end{equation}
If $e^{k_id_i}(2+k_iq_{i+1})\ge2+k_iq_i$ for every $i$, then
$M(t_n)\le q_n$ and $\Lambda\le t_n+q_n/2$.
\end{theorem}
\begin{proof}
At a highest simple zero $x+ih$, with $b<q=h^2\le q_i$, symmetry
and Lemma~\ref{lem:head} permit $x>X$. Its velocity is
$h'=-1/h-2\sum'K_h(d_j,v_j)$, where the sum omits its own
conjugate pair and real zeros have half-pair weight.
Every $v_j\le h$, so Lemma~\ref{lem:kernel-derivative} applies.
The full sums at height $p$ are $S$ and $\partial_pS$. They converge
absolutely and locally uniformly in $p$ above the zero strip, as do
the differentiated sums: paired summands are $O(d_j^{-2})$ there.
The removed pair contributes
\[
 K_p(0,h)=\frac{2p}{B-q},\qquad
 \partial_pK_p(0,h)=-\frac{2(B+q)}{(B-q)^2}.
\]
Consequently
\[
 \sum\nolimits'K_h\ge h\left\{
 \frac{\lambda_i}{p}\left(C_i-\frac{2p}{B-q}\right)
 -\mu_i\left(G_i+\frac{2(B+q)}{(B-q)^2}\right)\right\}
 \ge hk_i/4.
\]
The last bracket decreases with $q<B$. Thus $q'\le-2-k_iq$.
Induction, the strict barriers, compact attainment at positive time,
and backward vertical Hermite splitting are exactly as in
Theorem~\ref{thm:step}; at contact the classical strip theorem gives
$q\le q_i$, which is the polynomial certificate's range. If the
barrier reaches $b$ early, the strip theorem keeps $M\le b$ for
all later times. Otherwise the stated exponential gate gives
$M(t_{i+1})\le q_{i+1}$. The final strip theorem proves the conclusion.
\end{proof}

This is also a general conversion theorem: a zero-free half-plane
$\Rea s>\sigma$ supplies $q_0=(2\sigma-1)^2$ by the functional
equation. Any table satisfying the stated gates gives an explicit
bound $t_n+q_n/2$, independently of whether the seven-eighths
assertion is true. The numerical table below is certified for
$q_0=9/16$.

\begin{proof}[Proof of Theorem~\ref{thm:main}]
The half-plane hypothesis gives $M(0)\le9/16$ by \eqref{eq:M0}.
Use Theorem~\ref{thm:derivative-step} with $p=61/20$ and
Table~\ref{tab:kernel-steps}. Proposition~\ref{prop:two-density},
with $t^-=t_i$ and $t^+=t_{i+1}$, gives the required $C_i,G_i$;
Section~\ref{sec:kernel-checks} supplies all finite certificates.
The final time is $t_{24}=23/200+463/100000=11963/100000$,
and $q_{24}=1/25$. Hence
\[
 \Lambda\le\frac{11963}{100000}+\frac1{50}
          =\frac{13963}{100000}=0.13963.
\]
\end{proof}

\begin{table}[htbp]
\centering\small
\caption{Derivative comparison at $p=61/20$. Here $t_i=i/200$ for $0\le i\le23$,
$d_i=1/200$ for $i<23$, $d_{23}=463/100000$, and $q_{24}=1/25$.
Each next height is the following row's $q_i$, with the stated final height.}
\label{tab:kernel-steps}
\begin{tabular}{rrrrrr}
\toprule
$i$ & $10^6q_i$ & $10^4C_i$ & $10^4G_i$ & $10^3\lambda_i$ & $10^3\mu_i$\\
\midrule
0 & 562500 & 52902 & 2340 & 1718 & 868\\
1 & 528937 & 52964 & 2322 & 1729 & 888\\
2 & 496634 & 53031 & 2284 & 1739 & 907\\
3 & 465546 & 53100 & 2234 & 1749 & 927\\
4 & 435634 & 53172 & 2174 & 1759 & 945\\
5 & 406850 & 53245 & 2107 & 1769 & 964\\
6 & 379163 & 53318 & 2035 & 1778 & 982\\
7 & 352551 & 53390 & 1962 & 1791 & 1015\\
8 & 326987 & 53462 & 1887 & 1800 & 1034\\
9 & 302441 & 53533 & 1813 & 1808 & 1051\\
10 & 278890 & 53602 & 1739 & 1822 & 1088\\
11 & 256302 & 53670 & 1667 & 1830 & 1106\\
12 & 234657 & 53736 & 1596 & 1838 & 1123\\
13 & 213926 & 53800 & 1528 & 1845 & 1140\\
14 & 194091 & 53862 & 1462 & 1853 & 1156\\
15 & 175117 & 53922 & 1399 & 1866 & 1197\\
16 & 156984 & 53980 & 1338 & 1872 & 1213\\
17 & 139669 & 54036 & 1280 & 1879 & 1229\\
18 & 123142 & 54090 & 1225 & 1885 & 1244\\
19 & 107379 & 54143 & 1172 & 1892 & 1259\\
20 & 92351 & 54193 & 1122 & 1897 & 1273\\
21 & 78037 & 54242 & 1074 & 1910 & 1318\\
22 & 64407 & 54288 & 1028 & 1915 & 1332\\
23 & 51439 & 54334 & 985 & 1920 & 1345\\
\bottomrule
\end{tabular}
\end{table}

\section{Elementary checks and interpretation}\label{sec:checks}
For $x\ge0$, write
\[
 E_n(x)=\sum_{j=0}^n\frac{x^j}{j!},\qquad
 \overline E_n(x)=E_n(x)+\frac{x^{n+1}}{(n+1)!}
                  \left(1-\frac{x}{n+2}\right)^{-1}.
\]
When $x<n+2$, these bound $e^x$ below and above. The scalar inequalities
needed above can be checked without numerical integration:
\begin{gather*}
 \overline E_{160}(111/5)<\frac{7\cdot60000000000}{88},\\
 E_{20}\left(\frac{4675}{988}\frac4{25}\right)
       >\frac{73691}{34608}.
\end{gather*}
Use $\pi<22/7$ for the first. To bound $\zeta(5/2)$, the integral test
and integer square roots give the rational certificate
\[
 1+\sum_{n=2}^{64}
 \frac{10^{12}}{n^2\lfloor\sqrt{n\,10^{24}}\rfloor}
       +\frac{2}{3\cdot64\cdot8}<\frac{269}{200}.
\]
For the logarithm-weighted sum $J$, define the rational upper bound
\[
 A_{12}(r)=2\sum_{j=0}^{12}\frac{r^{2j+1}}{2j+1}
                  +\frac{2r^{27}}{27(1-r^2)}\qquad(0\le r\le1/3).
\]
The series for $2\operatorname{arctanh}r$ shows
$\log((1+r)/(1-r))\le A_{12}(r)$. For $2\le n\le64$, let
$2^{m_n}\le n<2^{m_n+1}$, set
$r_n=(n-2^{m_n})/(n+2^{m_n})$, and put
$\ell_n=m_n A_{12}(1/3)+A_{12}(r_n)$. Then $\log n\le\ell_n$,
and the integral test gives the exact rational certificate
\[
 J\le\sum_{n=2}^{64}
 \frac{10^{12}\ell_n}{n^2\lfloor\sqrt{n\,10^{24}}\rfloor}
 +\frac1{512}\left(\frac23\ell_{64}+\frac49\right)
 <\frac{39}{100}.
\]
The displayed upper bound is $0.3874050364\ldots$.
Also $E_{60}(12)>60000$, as used in the reflected derivative bound.
The lower bound for $S(x,4,t)$ in Proposition~\ref{prop:density}
then reduces to a rational inequality with value greater than
$5.25228>21/4$.

The threshold brackets in Remark~\ref{rem:halfplane} are certified similarly: at a
lower endpoint compare $(2+k a)/(2+k b)$ with
$E_{20}(k d)$, and at an upper endpoint compare it with
$\overline E_{20}(k d)$. Use $d=9/50$ for $F=1/5$ and
$d=1/5$ for the time-domain endpoint. All quantities in these checks
are rational.

For the Gaussian checks, replace $\log3$ by $11/10$, $\log2$ by $7/10$,
and $D^{-1/2}$ by $D^{-1}$ in the upper bounds. At
$A=30000000000-3$, $t=1/5$, the resulting rational-exponential expressions
are respectively less than $10^{-8}$, $3/10$, and $4/100$, as follows
from $\overline E_{100}$. These three inequalities, the two exponential
checks above, the derivative-sum certificate, and the threshold
brackets were checked in exact rational
arithmetic in preparing this note. The analytic estimates are proved in
the text; a scalar check alone would not establish them.

\subsection{Finite certificates for the stepped bounds}\label{sec:step-checks}
The stronger normalizer bound follows from the exact rational check
\[
 \overline E_{160}(557/25)<\frac{7X}{88},
\]
using $\pi<22/7$. For $u>1$ and $K_0=65536$, define
\begin{align*}
 T_0(u)&=\frac{K_0^{1-u}}{u-1},\\
 T_1(u)&=K_0^{1-u}
       \left(\frac{\log K_0}{u-1}+\frac1{(u-1)^2}\right),\\
 Z_*(u)&=\sum_{n=1}^{K_0}n^{-u}+T_0(u),\\
 J_*(u)&=\sum_{n=2}^{K_0}(\log n)n^{-u}+T_1(u),\\
 \psi_*(u)&=\sum_{\substack{\ell\ \text{prime},\ j\ge1\\\ell^j\le K_0}}
                       (\log\ell)\ell^{-ju}+T_1(u).
\end{align*}
For all the exponents used here, $u\ge\sigma_p>2$, so the integral
test and $0\le\Lambda_{\rm vM}(n)\le\log n$ give
$Z\le Z_*$, $J\le J_*$, and $\psi\le\psi_*$.
The same integrals with lower limit $60000$ or $65536$ verify the
four tail allowances in Proposition~\ref{prop:euler-step}.

Tables~\ref{tab:tri-cert} and~\ref{tab:eul-cert} give rational strict
upper bounds for these explicit finite expressions, and for
\eqref{eq:delta0}--\eqref{eq:delta1}. Each scalar column has denominator $10^{10}$; the column $i$
is the row index.
For the triangle table use $u_i=\sigma_p+2.78t_i$; for the Euler
table use $u_i=\sigma_p+5.57t_i$ and $t^+=t_{i+1}$ from the
corresponding step table. Thus, for example, $U_Z>Z_*(u_i)$ and
$U_0>\delta_0$; these are upper bounds rather than rounded
approximations.

\begin{table}[htbp]
\centering\small
\caption{Scalar upper bounds for the triangle table; scalar entries have denominator $10^{10}$.}\label{tab:tri-cert}
\begin{tabular}{rrr}
\toprule
$i$ & $10^{10}U_Z$ & $10^{10}U_J$\\
\midrule
0 & 16402710674 & 9276763593\\
1 & 16152183902 & 8753918709\\
2 & 15915614854 & 8271878376\\
3 & 15691927634 & 7826527146\\
4 & 15480153540 & 7414255565\\
5 & 15279418035 & 7031887314\\
6 & 15088929585 & 6676618291\\
7 & 14907970034 & 6345965473\\
8 & 14735886302 & 6037723783\\
9 & 14572083176 & 5749929570\\
10 & 14416017042 & 5480829566\\
11 & 14267190412 & 5228854407\\
12 & 14125147139 & 4992595958\\
13 & 13989468209 & 4770787832\\
14 & 13859768041 & 4562288608\\
15 & 13735691217 & 4366067315\\
\bottomrule
\end{tabular}
\end{table}

\begin{table}[htbp]
\centering\small
\caption{Scalar upper bounds for the Euler table; scalar entries have denominator $10^{10}$.}\label{tab:eul-cert}
\begin{tabular}{rrrrr}
\toprule
$i$ & $10^{10}U_Z$ & $10^{10}U_\psi$ & $10^{10}U_0$ & $10^{10}U_1$\\
\midrule
0 & 16402710674 & 5657220524 & 50075480 & 155020632\\
1 & 15914787749 & 5197368262 & 85920176 & 252268064\\
2 & 15478670974 & 4788551062 & 111472350 & 312846838\\
3 & 15086927154 & 4423194675 & 129334234 & 347974299\\
4 & 14733472069 & 4095166366 & 141433101 & 365502445\\
5 & 14413277797 & 3799434942 & 149205476 & 370949407\\
6 & 14122153062 & 3531827600 & 153728284 & 368204249\\
7 & 13856576220 & 3288851197 & 155813037 & 360012255\\
8 & 13613566708 & 3067557862 & 156074015 & 348311492\\
9 & 13390585030 & 2865442067 & 154978007 & 334467800\\
10 & 13185454126 & 2680360586 & 152880870 & 319439900\\
11 & 12996296991 & 2510469419 & 150054571 & 303896046\\
12 & 12821486753 & 2354173570 & 146707325 & 288296883\\
13 & 12659606398 & 2210086644 & 142998686 & 272954564\\
14 & 12509416032 & 2076998092 & 133116245 & 246627274\\
\bottomrule
\end{tabular}
\end{table}

We verified these finite sums with a fresh implementation using
outward-rounded Arb arithmetic at 160-bit precision.
Only exact integers and rationals were supplied as inputs. For a
purely rational alternative, write $n=2^m(1+r)/(1-r)$, $0\le r<1/3$,
and bracket $\log n$ by the first $33$ terms of the
$2\operatorname{arctanh}$ series and its geometric remainder, as in
$A_{12}$ above. Exponentials are then bracketed by $E_{160}$ and
$\overline E_{160}$; negative exponents are handled by reciprocation.
All positive exponential arguments needed for these finite sums are
below $40$. This gives a finite rational verification procedure for
each printed entry; it is not necessary to evaluate zeta or its
logarithmic derivative as a special function.

After these upper bounds are supplied, every remaining comparison
is rational. In particular the certified lower bounds for $S$ are
\begin{align*}
 C_i^{\rm tri,*}
 &=\frac{557}{100}-10^{-6}
              -\frac{\nu U_J+\epsilon_1}{2-U_Z-\epsilon_0},\\
 C_i^{\rm E,*}
 &=\frac{557}{100}-10^{-6}
     -\frac{\nu(U_\psi+U_1U_Z)+\epsilon_1U_Z}
                  {1-(U_0+\epsilon_0)U_Z}.
\end{align*}
The denominators are positive. The second expression follows by
multiplying numerator and denominator of \eqref{eq:euler-step} by
$Z$; its subtracted quotient is increasing in each nonnegative upper
bound $U_Z,U_\psi,U_0,U_1$ on this domain.
The exact rational comparisons with the step-table values give
\[
 \min_i(C_i^{\rm tri,*}-C_i)>8.5\cdot10^{-7},\qquad
 \min_i(C_i^{\rm E,*}-C_i)>1.48\cdot10^{-6}.
\]
Finally, with $k_i$ as in Theorem~\ref{thm:step}, exact arithmetic gives
\[
 \begin{array}{ll}
 \displaystyle\min_i\{E_{20}(k_id_i)(2+k_iq_{i+1})-(2+k_iq_i)\}
       >3.48\cdot10^{-7},&\text{triangle table},\\[3pt]
 \displaystyle\min_i\{E_{20}(k_id_i)(2+k_iq_{i+1})-(2+k_iq_i)\}
       >1.25\cdot10^{-7},&\text{Euler table}.
 \end{array}
\]
The chains, $p^2\ge5q_i$, $k_i>0$, $q_{i+1}\ge b$, and
$t_i+d_i\le1/5$ are also rational checks. The small margins explain
why the certificate tables retain ten decimal places.
No supplementary code or data are required to specify the checks.

\subsection{Certificates for the derivative comparison}\label{sec:kernel-checks}
All constants and finite formulas needed for the new bound are given
here and in Tables~\ref{tab:kernel-steps}--\ref{tab:kernel-reflection}.
No external table or verification program is needed to specify them.

\emph{Normalizer and small derivative terms.}
The sharper floor $L/4>5.5716$ and the cutoff $N\ge69098$ follow from
\[
 \overline E_{160}(13929/625)<\frac{113X}{1420},\qquad
 69098^2<\frac{113X}{1420},
\]
using $\pi<355/113$. Both displayed inequalities hold in exact
rational arithmetic. The weaker estimate $\pi<22/7$ used above
is not sharp enough for this new logarithm floor.
Also, the integral test and the logarithm series below give
\[
 J(2)\le\sum_{n=2}^{16}\frac{\log n}{n^2}
                  +\frac{\log16+1}{16}<1,
 \qquad \log N_1<12,\qquad
 \int_{N_1}^\infty(\log y)y^{-81/40}\,dy<10^{-3}.
\]
These establish the small-term estimates used when differentiating
$1/\zeta(s_B)$.

\emph{Kernel positivity on an unbounded domain.}
For row $i$ let $\Phi_i$ be \eqref{eq:kernel-poly} with
$A_2=q_i$, $\lambda=\lambda_i$, $\mu=\mu_i$, and $B=(61/20)^2$.
Compactify the first variable by setting
\begin{equation}\label{eq:compact-poly}
 \Psi_i(U,V)=\frac{(1-U)^4}{B^4}
              \Phi_i\left(\frac{BU}{1-U},q_iV\right).
\end{equation}
It extends to a polynomial of degree at most four in each variable
on $[0,1]^2$. For each $\ell=0,\ldots,7$, expand
\[
 \Psi_i((\ell+u)/8,v)=\sum_{k,m=0}^4a^{(i,\ell)}_{km}u^kv^m,
\]
and define its rational Bernstein coefficients by
\begin{equation}\label{eq:bernstein}
 \beta^{(i,\ell)}_{ab}
 =\sum_{\substack{0\le k\le a\\0\le m\le b}}
 a^{(i,\ell)}_{km}
 \frac{\binom ak}{\binom4k}\frac{\binom bm}{\binom4m},
 \qquad 0\le a,b\le4.
\end{equation}
For all $24\cdot8\cdot25=4800$ coefficients, exact rational
arithmetic gives the common bound
\begin{equation}\label{eq:bernstein-gate}
 \beta^{(i,\ell)}_{ab}>\frac{13}{10000}.
\end{equation}
The least coefficient is in row $21$ and equals
\[
 \frac{623287907570432169922330885087}
      {479268282493202500000000000000000}.
\]
The Bernstein basis polynomials
$\binom4a u^a(1-u)^{4-a}\binom4b v^b(1-v)^{4-b}$ are nonnegative
and sum to one. Thus \eqref{eq:bernstein-gate} proves positivity
on the eight rectangles, hence on all of $[0,1]^2$. Equation
\eqref{eq:compact-poly} proves $\Phi_i>0$ for every finite
$D\ge0$, $0\le r\le q_i$; no separate large-$D$ tail is needed.
Equations~\eqref{eq:kernel-poly}, \eqref{eq:compact-poly}, and
\eqref{eq:bernstein} specify each rational check directly from
Table~\ref{tab:kernel-steps}.

\emph{Finite scalar bounds.}
For $v>1$ and $j=0,1,2$, write
\[
 T_j(v)=N_1^{1-v}\sum_{k=0}^j
        \frac{j!}{(j-k)!}\frac{(\log N_1)^{j-k}}{(v-1)^{k+1}}
       =\int_{N_1}^\infty(\log y)^j y^{-v}\,dy.
\]
Use the finite upper expressions $Z_*(u)$ and $\psi_*(u)$ of
Section~\ref{sec:step-checks}, with $K_0=N_1$, and additionally
\[
 \psi_{2,*}(u)=
 \sum_{\substack{\ell\ \mathrm{prime},\ j\ge1\\\ell^j\le N_1}}
          j(\log\ell)^2\ell^{-ju}+T_2(u).
\]
The von Mangoldt inequality $\Lambda_{\rm vM}(n)\le\log n$ proves
$\psi_2(u)\le\psi_{2,*}(u)$. Every scalar entry in the following
two tables is a strict upper bound with denominator $10^{12}$.
Use $t^-=i/200$, $t^+=t^-+d_i$ from Table~\ref{tab:kernel-steps}.
The quantities bounded are $Z_*$, $\psi_*$, $\psi_{2,*}$,
the three expressions \eqref{eq:new-delta}, and the three expressions
\eqref{eq:new-rho}.

These bounds were evaluated independently using outward-rounded Arb
arithmetic at 160-bit precision, followed by upward rounding to the
printed rational values. They can alternatively be verified with
rational operations alone: range-reduce every logarithm to
$2\operatorname{arctanh}r$, $0\le r\le1/3$, and use
\[
 2\sum_{j=0}^{32}\frac{r^{2j+1}}{2j+1}
 \le2\operatorname{arctanh}r
 \le2\sum_{j=0}^{32}\frac{r^{2j+1}}{2j+1}
                     +\frac{2r^{67}}{67(1-r^2)}.
\]
Bracket positive exponentials by $E_{160}$ and
$\overline E_{160}$ and negative ones by reciprocation. All required
positive arguments are below $40$. The sums and tails are explicit;
no special-function evaluation of zeta is required.

\begin{table}[htbp]
\centering\small
\caption{Upper bounds for the three zeta quantities; scalar entries have denominator $10^{12}$.}\label{tab:kernel-zeta}
\begin{tabular}{rrrr}
\toprule
$i$ & $10^{12}U_Z$ & $10^{12}U_\psi$ & $10^{12}U_{\psi_2}$\\
\midrule
0 & 1622101766394 & 548568036655 & 839050132415\\
1 & 1598022583652 & 525889196940 & 790278947343\\
2 & 1575265183899 & 504513355618 & 745403434764\\
3 & 1553729318012 & 484337910062 & 704010822738\\
4 & 1533324564928 & 465270942310 & 665744767760\\
5 & 1513969156245 & 447229787842 & 630295662774\\
6 & 1495588965589 & 430139844962 & 597392956895\\
7 & 1478116636400 & 413933575813 & 566799005853\\
8 & 1461490826524 & 398549661570 & 538304098986\\
9 & 1445655551753 & 383932282888 & 511722399708\\
10 & 1430559613518 & 370030503029 & 486888602283\\
11 & 1416156098413 & 356797735912 & 463655155787\\
12 & 1402401939248 & 344191284922 & 441889941403\\
13 & 1389257528970 & 332171941144 & 421474315362\\
14 & 1376686380197 & 320703631812 & 402301449327\\
15 & 1364654824165 & 309753111476 & 384274914724\\
16 & 1353131743893 & 299289689704 & 367307468650\\
17 & 1342088337101 & 289284990179 & 351320007547\\
18 & 1331497905078 & 279712736940 & 336240661410\\
19 & 1321335664253 & 270548564135 & 322004006438\\
20 & 1311578577670 & 261769846278 & 308550378051\\
21 & 1302205203949 & 253355546416 & 295825269418\\
22 & 1293195561658 & 245286079995 & 283778803157\\
23 & 1284531007297 & 237543192549 & 272365265934\\
\bottomrule
\end{tabular}
\end{table}

\begin{table}[htbp]
\centering\small\setlength{\tabcolsep}{4pt}
\caption{Upper bounds for the heat perturbation and the reflected sum;
scalar entries have denominator $10^{12}$.}
\label{tab:kernel-perturbation}\label{tab:kernel-reflection}
\begin{tabular}{rrrrrrr}
\toprule
$i$ & $10^{12}U_{\delta_0}$ & $10^{12}U_{\delta_1}$ & $10^{12}U_{\delta_2}$ &
$10^{12}U_{\rho_0}$ & $10^{12}U_{\rho_1}$ & $10^{12}U_{\rho_2}$\\
\midrule
0 & 2333875727 & 6985223600 & 28309848180 & 5463325 & 138148683 & 3493305246\\
1 & 4322690915 & 12563261269 & 48875386058 & 4711723 & 119143253 & 3012723258\\
2 & 6024142228 & 17095933238 & 64730258347 & 4063714 & 102757352 & 2598380159\\
3 & 7476349853 & 20751263792 & 76752580144 & 3504997 & 88629315 & 2241130672\\
4 & 8712138109 & 23670175701 & 85661206261 & 3023246 & 76447487 & 1933094121\\
5 & 9759821971 & 25971028934 & 92045073509 & 2607840 & 65943299 & 1667479318\\
6 & 10643868624 & 27753363561 & 96386976633 & 2249626 & 56885302 & 1438433720\\
7 & 11385455210 & 29100992955 & 99082869065 & 1940716 & 49074031 & 1240913513\\
8 & 12002940105 & 30084565440 & 100457558756 & 1674312 & 42337574 & 1070571687\\
9 & 12512262033 & 30763690033 & 100777494499 & 1444553 & 36527767 & 923661632\\
10 & 12927278778 & 31188703974 & 100261199527 & 1246390 & 31516902 & 796954084\\
11 & 13260055254 & 31402145247 & 99087799059 & 1075469 & 27194903 & 687665582\\
12 & 13521108993 & 31439981634 & 97404000815 & 928038 & 23466885 & 593396825\\
13 & 13719619770 & 31332638410 & 95329817649 & 800863 & 20251056 & 512079571\\
14 & 13863608939 & 31105859202 & 92963265577 & 691154 & 17476907 & 441930883\\
15 & 13960093162 & 30781428303 & 90384225819 & 596509 & 15083653 & 381413710\\
16 & 14015216409 & 30377777736 & 87657623581 & 514855 & 13018888 & 329202919\\
17 & 14034363511 & 29910498276 & 84836047532 & 444404 & 11237433 & 284156032\\
18 & 14022258009 & 29392770247 & 81961910699 & 383617 & 9700331 & 245288004\\
19 & 13983046614 & 28835727251 & 79069234829 & 331164 & 8373993 & 211749493\\
20 & 13920372221 & 28248763665 & 76185125123 & 285902 & 7229458 & 182808132\\
21 & 13837437138 & 27639794960 & 73330990009 & 246841 & 6241750 & 157832391\\
22 & 13737057913 & 27015478364 & 70523550716 & 213131 & 5389332 & 136277669\\
23 & 13576392076 & 26288961540 & 67521313978 & 184036 & 4653631 & 117674307\\
\bottomrule
\end{tabular}
\end{table}

Insert these upper bounds in \eqref{eq:new-errors}--\eqref{eq:new-G},
replacing $Z,\psi,\psi_2,\delta_j,\rho_j$ by their printed bounds.
Call the resulting rational numbers $C_i^{\rm cert}$ and
$G_i^{\rm cert}$. All error expressions increase with these
nonnegative inputs; the subtracted expression for $C$ increases,
and the expression for $G$ increases. Hence
$C_i^{\rm cert}\le C^*$ and $G_i^{\rm cert}\ge G^*$.
Exact rational arithmetic gives
\begin{gather*}
 \max_i\varepsilon_{0,i}^{\rm cert}<\frac1{50},\qquad
 \min_i(C_i^{\rm cert}-C_i)>2.05\cdot10^{-6},\\
 \min_i(G_i-G_i^{\rm cert})>1.10\cdot10^{-6},\\
 \min_i\{E_{20}(k_id_i)(2+k_iq_{i+1})-(2+k_iq_i)\}
                                      >2.78\cdot10^{-8},
\end{gather*}
where $k_i$ is the exact rational expression \eqref{eq:new-k}.
The row chains and domain conditions also hold in exact arithmetic,
and $8.63<k_i<10.32$ for all rows. The small transition margin
explains why all integers are retained exactly. These certificates
verify the finite inequalities; the analytic justification remains
Proposition~\ref{prop:two-density} and the two comparison lemmas.

\section{Comparison and scope}
The published unconditional benchmark is $0.2$; the elementary
seven-eighths strip alone gives $9/32$. The following table compares
the published bound with the successive versions of our conditional
comparison. Gomila's separate claim $0.1787854$ has no conjectural
premise, but remains unrefereed; it is neither validated nor used here.
\begin{center}
\begin{tabular}{lll}
\toprule
Bound on $\Lambda$ & Additional hypothesis & Source or method\\
\midrule
$0.2$ & none & Platt--Trudgian, Corollary 2\\
$0.1938798366\ldots$ & $\Rea\rho\le7/8$ & height $5$, whole-function Cauchy\\
$0.1797275414\ldots$ & $\Rea\rho\le7/8$ & height $4$, fixed-cutoff derivative\\
$0.17172$ & $\Rea\rho\le7/8$ & height $301/100$, triangle steps\\
$0.16408$ & $\Rea\rho\le7/8$ & height $301/100$, Euler steps\\
$0.13963$ & $\Rea\rho\le7/8$ & height $61/20$, derivative kernel\\
\bottomrule
\end{tabular}
\end{center}
For completeness, the height-$5$ row follows from \eqref{eq:disk}:
Cauchy on the unit disk about $x+5i$ gives $|F'/F|\le7/13$, hence
$S(x,5,t)>111/20-10^{-6}-7/13>5$. Its half-plane formula is
\[
 k_5(a)=4-\frac8{25-a},\qquad
 F_5(\sigma)=\frac1{50}
  +\frac1{k_5(a)}\log\frac{2+k_5(a)a}{2+k_5(a)/25},
 \quad a=(2\sigma-1)^2>1/25.
\]
This older formula improves $0.2$ only for
$\sigma<0.8832895754\ldots$, and satisfies the required comparison-time
restriction only up to $\sigma=0.9104380067\ldots$.
Rational brackets are $(0.88328,0.88330)$ and $(0.91043,0.91045)$,
verified by the same exponential certificates as above. In particular
neither formula may be extrapolated to $\sigma=1$ to claim an
unconditional improvement.

The approximation and derivative arguments here were checked against
the stated source formulas, including the unrestricted real part in
Proposition 6.2 and the condition $K+u\ge2$. This is an internal
mathematical check, not independent refereeing or formal verification.
The published boundary conclusion \eqref{eq:wall} and the finite RH
verification remain inputs. Neither the seven-eighths preprint~\cite{O} nor
Gordon's preprint~\cite{G} is validated by this calculation. In particular,
the earlier stronger value $0.1259625$ is not established by the present
argument. The derivative comparison preserves the same approximation band
and published boundary input. It does not extend that band below $y=3$;
no lower-height or variable-height numerical improvement is claimed.

\section*{Acknowledgment of AI assistance}
The author proposed removing the dependence on the unpublished inputs.
ChatGPT assisted with the replacement strategy, analytic derivations,
exact scalar and Bernstein-coefficient checks, interval verification of the
stepped tables, and drafting. The author is responsible for review
of the mathematical argument and the final manuscript. This working
note has not been independently refereed.

\end{document}